\documentclass[12pt,reqno]{amsart}

\usepackage{amssymb,latexsym}

\usepackage{enumerate}

\usepackage[french,english]{babel}
\usepackage{amsmath}
\usepackage{graphicx}
\usepackage{amssymb}
\usepackage{bbm}
\usepackage{amsthm,mathtools}
\usepackage{ulem}
\usepackage{geometry}
\usepackage{tikz-cd}
\usepackage{mathrsfs}
\usepackage[colorinlistoftodos]{todonotes}
\usepackage{enumitem}
\usepackage{verbatim}
\usepackage[foot]{amsaddr}
\usepackage{dsfont}

\makeatletter

\@namedef{subjclassname@2010}{
	
	\textup{2020} Mathematics Subject Classification}

\makeatother
\newtheorem{thm}{Theorem}[section]
\newtheorem*{thm*}{Theorem}

\newtheorem{lem}[thm]{Lemma}

\theoremstyle{definition}

\numberwithin{equation}{section}

\usepackage{hyperref}
\hypersetup{hypertex=true,colorlinks=true,linkcolor=blue,anchorcolor=blue,citecolor=blue}
\newcommand{\newabstract}[1]{%
	\par\bigskip
	\csname otherlanguage*\endcsname{#1}%
	\csname captions#1\endcsname
	\item[\hskip\labelsep\scshape\abstractname.]
}

\begin{document}

	\baselineskip=17pt

	\title[Large values of character sums with multiplicative coefficients]{Large values of character sums with multiplicative coefficients}

	\author{Zikang Dong\textsuperscript{1}}
    \author{Zhonghua Li\textsuperscript{2}}
	\author{Yutong Song\textsuperscript{2}}
    \author{Shengbo Zhao\textsuperscript{2}}
	\address{1.School of Mathematical Sciences, Soochow University, Suzhou 215006, P. R. China}
	\address{2.School of Mathematical Sciences, Tongji University, Shanghai 200092, P. R. China}
    
	\email{zikangdong@gmail.com}
    \email{zhonghua.li@tongji.edu.cn}
	\email{99yutongsong@gmail.com}
	\email{shengbozhao@hotmail.com}
	\date{\today}
	
	\begin{abstract} 
		In this article, we investigate  large values of  Dirichlet character sums with multiplicative coefficients $\sum_{n\le N}f(n)\chi(n)$. We prove an Omega result in the region $\exp((\log q)^{\frac12+\varepsilon})\le N\le\sqrt q$, where $q$ is the prime modulus.
	\end{abstract}

	
	\maketitle
	
	\section{Introduction}
For any large prime number $q$ and any character $\chi({\rm mod}\;q)$, we always have the P\'olya-Vinogradov inequality for any $N>0$
   $$\sum_{n\le N}\chi(n)\ll\sqrt q\log Q,$$
   where $Q=q$ unconditionally and $Q=\log q$ on the generalized Riemann Hypothesis (GRH). The conditional upper bound is optimal up to the implied constant, since Paley showed for any large $q$, there always exists a quadratic character $\chi$ and $N<q$ such that
   $$\sum_{n\le N}\chi(n)\gg\sqrt q\log_2 q.$$
   We use $\log_j$ to denote the $j$-th iteration of logarithm. When $N$ is fixed, consider the maximum
   $$S_q(N):=\max_{\chi_0\neq\chi({\rm mod }\;q)}\Big|\sum_{n\le N}\chi(n)\Big|.$$
   Granville and Soundararajan \cite{GS01} conjectured $S_q(x)$ increases on $N<\sqrt q$. They showed several evidence to believe this. That is, lower bounds for $S_q(N)$ according to the range of $N$. See Theorems 3--8 of \cite{GS01}. Part of these have been improved separately by Munsch \cite{Mun}, Hough \cite{Hou} and La Bret\`eche and Tenenbaum \cite {BT}.These results are mainly divided into two cases: the case $x$ is less than $\exp(\sqrt{\log q})$, and the case  x is between $\exp((\log q)^{\frac12+\varepsilon})$ and $\sqrt q$. When $x>\sqrt q$, we can use the ``Fourier flips" for character sums to transfer to the case of $x<\sqrt q$.

 This article focuses on the values of the following sums with multiplicative coefficients
 $$\sum_{n\le N}f(n)\chi(n),$$
 where $f(n)$ is a multiplicative function. We may expect this sum can be as large as $\sum_{n\le N}\chi(n)$.
 Harper \cite{Harper23} showed that if $|f(p)|=1$ and $|f(p^j)|\le1$ for any prime $p$ and integer $j$, then the low moments ($0<k<1$ fixed) can be bounded with better than square-root cancellation:
 $$\frac{1}{\varphi(q)}\sum_{\chi({\rm mod}\;q)}\Big|\sum_{n\le N}f(n)\chi(n)\Big|^{2k}\ll \bigg(\frac{N}{\sqrt{\log_2(\min\{N,q/N\})}}\bigg)^k.$$
 Harper conjectured that, if $f(\cdot)$ is not always 1, then the minimum $\min\{N,q/N\}$ in the above inequality can be replaced by $N$. This means the ``Fourier flips" for character sums no longer influence the mix sums when $\sqrt q<N<q$. Thus the character sums with multiplicative coefficients  also play important roles in analytic number theory as character sums.
\begin{thm}\label{thm1.1}
 Fix $\delta \in(0,\frac1{100})$ be any fixed small number. Let $q$ be sufficiently large prime and $\exp((\log q)^{\frac12+\delta})\le N\le q^{\frac12}$ . Let $f(\cdot)$  be any completely multiplicative function such that $|f(n)|=1$ for any integer $n$. Then we have
    $$ \max_{\chi({\rm mod}\;q)\atop \chi\neq\chi_0}\Big|\sum_{n\le N}f(n)\chi(n)\Big|\ge \sqrt{N}\exp\bigg((1+o(1)){\sqrt{\frac{\log (q/N)}{\log_2 (q/N)}}}\bigg).$$ 
\end{thm}
 Our theorem can be seen as a generalization of Theorem 3.1 of Hough's work \cite{Hou}, who showed it with $f(\cdot)=1$. Recently, Xu and Yang \cite{XY} also showed a similar result with $\chi(n)$ replaced by $n^{it}$. 
\section{Preliminary lemmas}
We need the following lemma which is the main technique in operating the resonance method.
\begin{lem}\label{lem2.1}
    Let $X$ be large and $\lambda=\sqrt{\log X\log_2 X}$. Define the multiplicative function $r(\cdot)$  supported on square-free integers:
$$r(p)=\begin{cases}
   \frac{\lambda}{\sqrt q \log p}, &  \lambda\le p\le \exp((\log\lambda)^2),\\
   0, & {\rm otherwise.}
\end{cases}$$
If $\log N>3\lambda\log_2\lambda$, then we have
\begin{equation*}\label{DD}
    \sum_{\substack{m,n\le N\\a,b\le X\\an=bm}}r(a)r(b)\Big/\sum_{n\le X}r(n)^2\ge N\exp{\bigg((2+o(1))\sqrt{\frac{\log X}{\log_2 X}}}\bigg).
\end{equation*}
\end{lem}
\begin{proof}
    This follows directly from Eq. (2.8) of \cite{XY}, which is based on the work of \cite{Hou}.
\end{proof}

        \section{Proof of Theorem \ref{thm1.1}}
Our proof employs the resonance method, which can date back to the work of Hilberdink \cite{HIL}, and was generally developed by Soundararajan\cite{Sound}.   For ease of representation, we denote
$$D_{N,f}(\chi):=\sum_{n\le N}f(n)\chi(n),$$
where $f$ is a completely multiplicative function with $|f(n)|=1$ for any $n\in\mathbb{N}$.
As in \cite{Sound}, we define the ``resonator''  $$R_\chi := \sum_{n\le X}r_f(n)\chi(n),$$where $r_f(n)$ is a function of the integer $n$, and $X:=q/N$. 
Then, set
$$M_1=M_1(R,q):=\sum_{\substack{\chi({\rm mod}\;q)\\\chi \neq \chi_0}}|R_\chi|^2,$$
and
$$M_2=M_2(R,q):=\sum_{\substack{\chi({\rm mod}\;q)\\\chi \neq \chi_0}}|D_{N,f}(\chi)|^2|R_\chi|^2.$$
Applying the idea of the resonance method, we obtain
\begin{equation}\label{resonnance}
    \max_{\chi({\rm mod}\;q)\atop \chi \neq \chi_0}\Big|\sum_{n\le N}f(n)\chi(n)\Big|\ge \sqrt{\frac{M_2}{M_1}}.
\end{equation}
Therefore, our main objective is to compute the ratio $M_2/M_1$. First, by the orthogonality of Dirichlet characters, we have
\begin{align*}
    M_1&=\sum_{\substack{\chi({\rm mod}\;q)\\\chi \neq \chi_0}}\sum_{a,b\le X}r_f(a)\overline{r_f(b)}\chi(a)\overline{\chi(b)}\\&\le \varphi(q)\sum_{\substack{a,b\le X\\a\equiv b({\rm mod}\;q)}}r_f(a)\overline{r_f(b)}.
\end{align*}
Since $0<X<q$, we can get $a=b$. Thus,
\begin{equation}
    M_1\le \varphi(q)\sum_{n\le X}|r_f(n)|^2.\label{m1}
\end{equation}
Expanding $M_2$, we have
\begin{align*}
    M_2&=\sum_{\substack{\chi({\rm mod}\;q)\\\chi \neq \chi_0}}|D_{N,f}(\chi)|^2|R_\chi|^2+|D_{N,f}(\chi_0)|^2|R_{\chi_0}|^2-|D_{N,f}(\chi_0)|^2|R_{\chi_0}|^2
    \\&=\varphi(q)\sum_{\substack{m,n\le N\\a,b\le X\\an\equiv bm({\rm mod}\;q)}}f(n)\overline{f(m)}r_f(a)\overline{r_f(b)}-|D_{N,f}(\chi_0)|^2|R_{\chi_0}|^2.
\end{align*}
Then $NX\le q$ impies that $an,bm\le q.$ Hence $an\equiv bm({\rm mod}\;q)$ reduces to the only case $an=bm.$ Consider the part concerning the principal character. By Cauchy's inequality, we have
$$|R_{\chi_0}|^2\le X\sum_{n\le X}|r_f(n)|^2.$$
And similarly we can get
$$|D_{N,f}(\chi_0)|^2=\Big|\sum_{n\le N}f(n)\chi_0(n)\Big|^2 \le N\sum_{n\le N}|f(n)\chi_0(n)|^2\le N^2.$$
Combining the above together we get
\begin{equation*}
    M_2=\varphi(q)\sum_{\substack{m,n\le N\\a,b\le X\\an=bm}}f(n)\overline{f(m)}r_f(a)\overline{r_f(b)}+O(qN)\sum_{n\le X}|r_f(n)|^2.
\end{equation*}
Now we choose $$r_f(n):=f(n)r(n),$$where $r(n)$ is a non-negative multiplicative function. Since $|f(n)|=1$, we can obtain
\begin{equation*}
    M_2=\varphi(q)\sum_{\substack{m,n\le N\\a,b\le X\\an=bm}}r(a)r(b)+O(qN)\sum_{n\le X}|r_f(n)|^2.
\end{equation*}
Combined with \eqref{m1}, we have
\begin{equation*}\label{ratio}
    \frac{M_2}{M_1}\ge \sum_{\substack{m,n\le N\\a,b\le X\\an=bm}}r(a)r(b)\Big/\sum_{n\le X}r(n)^2+O(N).
\end{equation*}
Put $\lambda=\sqrt{\log X\log_2 X}$ and choose the multiplicative function $r(\cdot)$ to be supported on square-free integers satisfying
$$r(p)=\begin{cases}
   \frac{\lambda}{\sqrt q \log p}, &  \lambda\le p\le \exp((\log\lambda)^2),\\
   0, & {\rm otherwise,}
\end{cases}$$
as in Lemma \ref{lem2.1}. The range $\exp((\log q)^{\frac12+\delta})\le N\le q^{\frac12}$ and $X=q/N$ guarantee the condition $\log N>3\lambda\log_2\lambda$. Thus by Lemma \ref{lem2.1} we have
\begin{equation*}\label{DD}
     \frac{M_2}{M_1}\ge\sum_{\substack{m,n\le N\\a,b\le X\\an=bm}}r(a)r(b)\Big/\sum_{n\le X}r(n)^2\ge N\exp{\bigg((2+o(1))\sqrt{\frac{\log X}{\log_2 X}}}\bigg)+O(N).
\end{equation*}
At last, back to \eqref{resonnance} we can obtain
$$\max_{\substack{\chi({\rm mod}\;q)\\ \chi \neq \chi_0}}\Big|\sum_{n\le N}f(n)\chi(n)\Big|\ge \sqrt{\frac{M_2}{M_1}}\ge \sqrt{N}\exp\bigg((1+o(1)){\sqrt{\frac{\log (q/N)}{\log_2 (q/N)}}}\bigg). $$
\qed
	\section*{Acknowledgements}
	The authors are supported by the Shanghai Magnolia Talent Plan Pujiang Project (Grant No. 24PJD140) and the National
	Natural Science Foundation of China (Grant No. 	1240011770).

	\normalem

\end{document}